\documentclass[12pt]{article}
\usepackage[margin=1in]{geometry}
\usepackage{graphicx} 
\usepackage{algorithm}
\usepackage{algpseudocode}
\usepackage{amsmath}
\usepackage{amsthm,physics}
\usepackage{amssymb}
\usepackage{xcolor}
\usepackage{hyperref}
\usepackage{mathtools}
\usepackage{comment}
\usepackage{subcaption}
\usepackage{bm,authblk,enumerate}
\usepackage{pgf}
\usepackage{booktabs}
\usepackage{censor}

\theoremstyle{plain}
\newtheorem{theorem}{Theorem}[section]

\newtheorem{proposition}[theorem]{Proposition}
\newtheorem{corollary}[theorem]{Corollary}

\theoremstyle{definition}

\theoremstyle{remark}

\newtheorem{remark}[theorem]{Remark}

\numberwithin{equation}{section}

\usepackage{dsfont}

\newcommand{\set}[1]{\left\{#1\right\}}
\newcommand{\R}{\mathbb{R}}
\newcommand{\Z}{\mathbb{Z}}
\newcommand{\cD}{\mathcal{D}}
\newcommand{\cE}{\mathcal{E}}
\newcommand{\cL}{\mathcal{L}}
\newcommand{\cP}{\mathcal{P}}
\newcommand{\cV}{\mathcal{V}}

\newcommand{\bmu}{{\bm{\mu}}}
\newcommand{\blambda}{{\bm{\lambda}}}
\newcommand{\bc}{{\vb{c}}}
\newcommand{\bF}{{\vb{F}}}
\newcommand{\bG}{{\vb{G}}}
\newcommand{\bn}{{\vb{n}}}

\newcommand{\bw}{{\vb{w}}}
\newcommand{\bx}{{\vb{x}}}
\newcommand{\bz}{{\vb{z}}}

\usepackage{etoolbox}
\newlength\Normalbaselineskip
\appto\MultlinedHook{\setlength\baselineskip{\Normalbaselineskip}}

\usepackage[utf8]{inputenc}
\DeclareFixedFont{\ttb}{T1}{txtt}{bx}{n}{12} 
\DeclareFixedFont{\ttm}{T1}{txtt}{m}{n}{12}  
\usepackage{color}
\definecolor{deepblue}{rgb}{0,0,0.5}
\definecolor{deepred}{rgb}{0.6,0,0}
\definecolor{deepgreen}{rgb}{0,0.5,0}
\usepackage{listings}
\newcommand\pythonstyle{\lstset{
language=Python,
basicstyle=\ttm,
morekeywords={self},              
keywordstyle=\ttb\color{deepblue},
emph={MyClass,__init__},          
emphstyle=\ttb\color{deepred},    
stringstyle=\color{deepgreen},
frame=tb,                         
showstringspaces=false
}}
\lstnewenvironment{python}[1][]
{
\pythonstyle
\lstset{#1}
}
{}

\newcommand\pythoninline[1]{{\pythonstyle\lstinline!#1!}}

\allowdisplaybreaks

\begin{document}

\title{Upper bounds for the connective constant of weighted self-avoiding walks}
\author{Qidong He}
\affil{Department of Mathematics, Rutgers University}
\date{}

\maketitle

\begin{abstract}
Building on a work by Alm, we consider a model of weighted self-avoiding walks on a lattice and develop a method for computing upper bounds on the corresponding weighted connective constant, which we implement in a publicly available software package.
The upper bounds are obtained as the dominant eigenvalues of certain matrices and provide detailed information about the convergence of the model's multivariate generating function. 
We discuss potential applications of our results to developing Peierls-type estimates for anisotropic contour models in statistical physics, generalizing a technique recently introduced by Fahrbach–Randall.
\end{abstract}

\section{Introduction}

Self-avoiding structures on lattices and their associated growth constants have been a long-standing subject of study across many disciplines \cite{flory1953principles,herrero2005self,michelen2025potential}.
The most prominent example of these is the self-avoiding walk (SAW).
In a seminal work from 1957 \cite{hammersley1957percolation}, Hammersley showed that the growth rate of the number of SAWs originating from any given vertex is, for a large class of lattices, governed by a finite constant known as the \emph{connective constant} of the lattice.
Since then, the connective constant has attracted considerable interest, but its exact value has been determined only in a minimal number of cases, including on the honeycomb lattice in a breakthrough by Duminil-Copin--Smirnov \cite{duminil2012connective} and in a subsequent generalization of their work by Glazman \cite{glazman2015connective}.
For other lattices, including popular instances such as the square lattice in 2D, the study of the connective constant, and SAWs in general, is primarily limited to developing better approximation or sampling methods, perhaps by deriving rigorous lower \cite{kesten1963number,beyer1972lower} and upper \cite{alm1993upper} bounds on the connective constant or through Monte Carlo simulation \cite{madras1988pivot}.

It is also possible to define and study \emph{weighted} SAWs and their associated connective constants.
Such generalizations appear naturally, for instance, in modeling polymer chains interacting with an external electric field \cite{borgs2000anisotropic}, as well as in the analysis of specific lattice models in statistical physics \cite{malakis1979modified,fahrbach2019slow}.
However, there does not appear to be a concerted effort dedicated to the study of weighted SAWs and their connective constants, as existing literature generally focuses on highly specialized aspects of the matter \cite{madras1999pattern,rechnitzer2006haruspicy,james2007new,betz2019scaling,grimmett2020weighted}.
We may partly attribute the lack of \emph{general} results to an inherent difficulty with reporting and comparing results on weighted SAWs, namely their dependence on the weighting scheme, which is context-dependent and difficult to standardize.

This paper aims to bridge a methodological gap in the literature on weighted SAWs by supplying a general method for deriving upper bounds on the weighted connective constant.
To ensure that our results are both reproducible and readily applicable, we have implemented the method in a publicly available software package.
Our definition of the weighted connective constant follows the approach of Hammersley \cite{hammersley1957percolation} and applies to SAWs on a large class of lattices.
Our method for upper-bounding the weighted connective constant is, in turn, based on a work by Alm \cite{alm1993upper}, which produced state-of-the-art bounds for the (unweighted) connective constant on many lattices at the time of its publication.
Specifically, we consider:
\begin{itemize}
    \item the same class of lattices that Alm considered (see Conditions \ref{cond:lattice}--\ref{cond:strongly-connected}), which forms a proper subclass of lattices to which Hammersley's work applies, but for which the results are simpler to state, and which is large enough to include popular structures such as the square and hexagonal lattices in 2D and the simple cubic lattice in 3D;
    \item weighting schemes specified by first assigning a translation-invariant weight to the edges of the lattice, and defining the weight of an SAW \emph{multiplicatively} as the product of the weights of the edges it contains.
\end{itemize}
See Sections \ref{sec:assumptions} and \ref{sec:notation} for precise formulations of these assumptions, Section \ref{sec:definition} for the definition of the weighted connective constant, and Section \ref{sec:algorithm} for descriptions of the method.
We note that our method works without modification for self-avoiding trails (SATs), and that extensions to other weighting schemes are possible but not pursued in this paper.

Our interest in deriving upper bounds on the weighted connective constant is strongly motivated by their potential applications in statistical physics.
It is well-known that upper bounds on the (unweighted) connective constant correspond to lower bounds on the radius of convergence of the generating function for SAWs \cite[(1.3.5)]{madras2013self}.
Recently, Fahrbach--Randall realized in their work on the six-vertex model \cite{fahrbach2019slow} that they could use a \emph{weighted} analog of this fact to develop a Peierls-type argument for a contour model on $\Z^2$ where the horizontal edges are weighted \emph{differently} than the vertical ones.
In this regard, our results in Sections \ref{sec:data} and \ref{sec:generating-function} \emph{systematize} their approach and allow for the development of analogous, improved Peierls-type estimates for contour models on general lattices.
We elaborate on the latter possibility in Section \ref{sec:contour-model}, where we use similar ideas to prove the convergence of the cluster expansion for a particular toy model.

\section{Method}
\label{sec:method}

\subsection{Assumptions on the lattice}
\label{sec:assumptions}

In this paper, a \emph{lattice} is a (possibly directed) infinite graph, embedded in a finite-dimensional Euclidean space, whose drawing forms a regular tiling.
As in \cite{alm1993upper}, we impose the following regularity constraints on the structure of the lattice:
\begin{enumerate}[I.]
    \item There is a finite number $K$ of equivalence classes of vertices under the translations of the lattice. 
    \label{cond:lattice}
    \item The (out-)degree of any vertex is finite.
    \label{cond:bounded-degree}
    \item The lattice is strongly connected, i.e., any two vertices are connected by a walk.
    \label{cond:strongly-connected}
\end{enumerate}
Examples of lattices to which our method applies include the square, triangular, and hexagonal lattices in 2D, and the simple, body-centered, and face-centered cubic lattices in 3D.

\subsection{Notation}
\label{sec:notation}

Throughout the paper, let $\cL$ be a lattice satisfying Conditions \ref{cond:lattice}--\ref{cond:strongly-connected}.
For $n\ge 1$, an $n$-step \emph{walk} on $\cL$ is a sequence $(v_0,e_1,v_1,\dots,e_n,v_n)$ of $n$ edges and $n+1$ vertices such that $e_i$ connects $v_{i-1}$ and $v_{i}$ for all $1\le i\le n$.
We call $v_0$ the \emph{starting point} of the walk.
We say that the walk is a \emph{self-avoiding walk} (SAW) if the vertices $v_0,\dots,v_n$ are distinct, or a \emph{self-avoiding trail} (SAT) if the edges $e_1,\dots,e_n$ are distinct.
In the sequel, we will mostly speak of SAWs to avoid repetition, but all of our results hold without modification for SATs.

We consider a model of \emph{weighted} SAWs, defined as follows.
Denote by $\cE_1,\dots,\cE_d$ the equivalence classes of edges of $\cL$ under its translations.
Given a function $\bz:\set{\cE_1,\dots,\cE_d}\rightarrow\R_{>0}$ or, equivalently, a vector $\bz\in\R_{>0}^{d}$, define the \emph{weight} (with respect to $\bz$) of an SAW $\gamma$ as
\begin{equation}
    \bw(\gamma):=\prod_{i=1}^d \bz_i^{N_i(\gamma)},
\end{equation}
where $N_i(\gamma)$ is the number of edges of $\gamma$ belonging to $\cE_i$ for $1\le i\le d$.
Note that the weight function $\bw(\cdot)$ thus defined is again invariant under the translations of $\cL$.

Recall the definition of $K$ from Condition \ref{cond:lattice}.
Denote by $\cV_1,\dots,\cV_K$ the equivalence classes of vertices of $\cL$ under its translations.
For $1\le k\le K$, $n\ge 1$, and $\bn\in\Z_{\ge 0}^d$, define
\begin{itemize}
    \item $\bc_{n,k}$ as the \emph{sum} of the weights of all $n$-step SAWs starting from a fixed vertex in $\cV_k$;
    \item $c_{\bn,k}$ as the \emph{number} of SAWs starting from a fixed vertex in $\cV_k$ which contains exactly $\bn_i$ edges of $\cE_i$ for each $1\le i\le d$ (this quantity will only be used in Section \ref{sec:generating-function}).
\end{itemize}
By translation invariance, both quantities are well-defined.

\subsection{The weighted connective constant}
\label{sec:definition}

Following Hammersley \cite{hammersley1957percolation}, we define the weighted connective constant, as follows: with
\begin{equation}
    \bc_n:=\max_{1\le k\le K}\bc_{n,k},
\end{equation}
the \emph{weighted connective constant} of SAWs (with respect to $\bz$) on $\cL$ is given by
\begin{equation}
    \bmu:=\inf_{n\ge 1}\bc_n^{1/n}.
\end{equation}
The weighted connective constant may be equivalently defined by a familiar limit used in the usual definition of the (unweighted) connective constant, as follows; importantly, this limit can be made independent of the starting point of SAWs.

\begin{theorem}
\label{thm:limit-definition}
    For all $\bz\in\R_{>0}^{d}$, the weighted connective constant $\bmu$ can be computed as 
    \begin{equation}
    \label{eqn:limit-definition}
        \bmu
        =\lim_{n\to\infty}\bc_{n}^{1/n}
        =\lim_{n\to\infty}\bc_{n,k}^{1/n}
    \end{equation}
    for any $1\le k\le K$.
\end{theorem}

The proof of Theorem \ref{thm:limit-definition} is an adaptation of Hammersley's argument \cite{hammersley1957percolation} to the weighted case.
As it is somewhat lengthy and not directly relevant to the description of our computational method (albeit essential in several proofs to follow), we defer it to Appendix \ref{appx:Hammersley}.

\subsection{Extending Alm's algorithm}
\label{sec:algorithm}

Having defined the weighted connective constant $\bmu$, we now fulfill the primary goal of the paper, namely the derivation of upper bounds on $\bmu$.
Our method is an adaptation of Alm's \cite{alm1993upper} to the weighted case, which boils down to considering a larger class of walks than the SAWs for which the limit in \eqref{eqn:limit-definition} can be computed explicitly.

Let $\bz\in\R_{>0}^d$ and $m\ge 0$.
We say that a symmetry $\tau$ of $\cL$ is \emph{weight-preserving} (with respect to $\bz$) if
\begin{equation}
    \bw(e)=(\bw\circ\tau)(e)\quad\text{ for all edges }e\text{ of }\cL.
\end{equation}
Let $\set{v_1,\dots,v_K}$ be a set of representatives from the vertex classes $\cV_1,\dots,\cV_K$, with $v_k\in\cV_k$ for each $1\le k\le K$ (see Section \ref{sec:notation}).
Let $\cP=\set{\cP_1(m),\dots,\cP_t(m)}$ be a partition of the set of all $m$-step SAWs starting from any of the $v_k$, $1\le k\le K$, such that, for each $1\le s\le t$, any two walks in $\cP_s(m)$ are related by a weight-preserving symmetry of the lattice.
In particular, we consider a single vertex as a $0$-step SAW with weight $1$.

For $n>m$ and $1\le r,s\le t$, define $\bF^\cP_{rs}(m,n)$ as the sum of the weights of all $n$-step SAWs that start with a fixed $m$-step SAW in $\cP_r(m)$ and end with a translation of any $m$-step SAW in $\cP_s(m)$.
Moreover, define 
\begin{equation}
\label{eqn:G-matrix-element}
    \bG^\cP_{rs}(m,n):=\frac{\bF^\cP_{rs}(m,n)}{\bw(\gamma_r(m))},
\end{equation}
where $\gamma_r(m)\in\cP_r(m)$ is arbitrary.
Let $\bG^\cP(m,n)$ be the $t\times t$ matrix with matrix elements $\bG^\cP_{rs}(m,n)$:
\begin{equation}
\label{eqn:G-matrix}
    \bG^\cP(m,n):=[\bG^\cP_{rs}(m,n)]_{1\le r,s\le t}.
\end{equation}
Note that all of the above quantities are well-defined.

For a square matrix $A$, we use the entry-wise $1$-norm $\norm{A}:=\sum_{i,j}\abs{A_{ij}}$.
In the case that $A$ is \emph{primitive}, i.e., $A$ is nonnegative and some positive-integer power of $A$ is positive, we denote by $\lambda_1(A)$ its dominant eigenvalue, which exists uniquely and is positive by the Perron--Frobenius theorem \cite[Theorem 8.4.4]{horn2012matrix}.

The following theorem is our main result, which establishes an upper bound on $\bmu$ in terms of the dominant eigenvalue of $\bG^\cP(m,n)$, assuming that the matrix is primitive.
It combines the ideas of \cite[Theorems 1 and 2]{alm1993upper} into a single, unified result, while allowing us to circumvent some delicate technicalities in practical applications of the theorem that Alm did not address in detail; see Remark \ref{rem:primitivity} after the proof of the theorem.

\begin{theorem}[Upper bound on $\bmu$]
\label{thm:upper-bound}
    If $\bG^\cP(m,n)$ is primitive, then 
    \begin{equation}
        \bmu\le\lambda_1(\bG^\cP(m,n))^{1/(n-m)}.
    \end{equation}
\end{theorem}

\begin{proof}
    Let $1\le r,s\le t$.
    Consider two $n$-step SAWs, one starting with a fixed $\gamma_r(m)\in\cP_r(m)$ and ending with some $m$-step SAW $\gamma'(m)$, and another starting with the same $\gamma'(m)$ and ending with a translation of any $\gamma_s(m)\in\cP_{s}(m)$. 
    Joining these SAWs at $\gamma'(m)$ yields a $(2n-m)$-step, but not necessarily self-avoiding, walk that starts with $\gamma_r(m)\in\cP_r(m)$ and ends with a translation of $\gamma_s(m)\in\cP_s(m)$.
    Since every $(2n-m)$-step SAW that starts with $\gamma_r(m)$ and ends with a translation of some $m$-step SAW in $\cP_s(m)$ can be obtained in this way, we may upper bound
    \begin{equation}
        \bG^\cP_{rs}(m,2n-m)
        \le\frac{1}{\bw(\gamma_r(m))}
        \sum_{u=1}^t \frac{\bF^\cP_{ru}(m,n)\bF^\cP_{us}(m,n)}{\bw(\gamma_u(m))}
        =[\bG^\cP(m,n)^2]_{rs},
    \end{equation}
    where $\gamma_u(m)\in\cP_u(m)$ for $1\le u\le t$, and the definition of $\bF^\cP_{ru}(m,n)$ and the summation over $u$ cover all the possibilities of $\gamma'(m)$.
    Proceeding by induction, we find that, for all $q\ge 2$,
    \begin{equation}
        \bG_{rs}^\cP(m,m+q(n-m))\le[\bG^\cP(m,n)^q]_{rs}.
    \end{equation}
    Hence, $\norm{\bG^\cP(m,m+q(n-m))}\le\norm{\bG^\cP(m,n)^q}$.
    Notice that
    \begin{equation}
    \begin{multlined}
        \sum_{k=1}^K\bc_{m+q(n-m),k}
        =\sum_{r=1}^t \abs{\cP_r(m)} \sum_{s=1}^t \bF^\cP_{rs}(m,m+q(n-m))
        \\
        \le\left[\max_{1\le r\le t}\abs{\cP_r(m)}\bw(\gamma_r(m))\right]
        \norm{\bG^\cP(m,m+q(n-m))}.
    \end{multlined}
    \end{equation}
    Therefore,
    \begin{equation}
    \begin{multlined}
        \bmu=\lim_{q\to\infty}\bc_{m+q(n-m)}^{1/[m+q(n-m)]}
        \le \lim_{q\to\infty}\left[\sum_{k=1}^K\bc_{m+q(n-m),k}\right]^{1/[m+q(n-m)]}
        \\
        \le \lim_{q\to\infty}\left[\max_{1\le r\le t}\abs{\cP_r(m)}\bw(\gamma_r(m))\right]^{\frac{1}{m+q(n-m)}}
        \lim_{q\to\infty}\norm{\bG^\cP(m,n)^q}^{\frac{1}{m+q(n-m)}}
        =\lambda_1(\bG^\cP(m,n))^{\frac{1}{n-m}},
    \end{multlined}
    \end{equation}
    where we used Theorem \ref{thm:limit-definition} in the first equality and the assumption that $\bG^\cP(m,n)$ is primitive in the last.
\end{proof}

\begin{remark}[Form of the matrix]
    The main difference between our matrix $\bG^\cP(m,n)$ and the ones Alm considered in the unweighted case is the normalization factor $1/\bw(\gamma_i(m))$ in \eqref{eqn:G-matrix-element}.
    This factor is essential for ensuring that the weight of the joined walk is computed correctly in the proof of Theorem \ref{thm:upper-bound}, by accounting for the weight of the $m$-step overlap.
    In the case that $\bw\equiv 1$, which is equivalent to having unweighted SAWs, our matrix correctly reduces to Alm's with suitable choices of $\cP$.
\end{remark}

\begin{remark}[Generality]
    Theorem \ref{thm:upper-bound} is more general than \cite[Theorem 2]{alm1993upper}:
    whereas Alm \emph{fixed} the sets $\cP_1(m),\dots,\cP_t(m)$ to be the \emph{equivalence classes} of $m$-step SAWs under all (weight-preserving) symmetries of the lattice, we allow for a possibly finer partition.
    While this may seem superfluous, as a finer partition $\cP$ gives rise to a larger matrix $\bG^\cP$, the generalization carries one crucial advantage.
    Namely, when we consider the entries of $\bz$ symbolically, we do not have to account for relations between the variables in $\bz$, which can affect the set of available weight-preserving symmetries.
    For example, on the square lattice, we can place SAWs related by the $x\leftrightarrow y$ reflection in different sets in the partition $\cP$, and the resulting bound would be valid for all $\bz\in\R_{>0}^2$, whereas a coarser partition that groups together these SAWs would require that $\bz_1=\bz_2$ for the $x\leftrightarrow y$ symmetry to be weight-preserving.
    Balancing the trade-off between the dimension of $\bG^\cP(m,n)$ and the set of vectors $\bz$ to which the upper bound applies is, therefore, best handled on a case-by-case basis.
\end{remark}

\begin{remark}[Assumption of primitivity]
\label{rem:primitivity}
    The assumption of primitivity is as essential to Theorem \ref{thm:upper-bound} as to \cite[Theorems 1 and 2]{alm1993upper}, all of which rely crucially on the existence of the dominant eigenvalue and its computability via the power method.
    In practice, Alm noted without proof in \cite[Remark 5]{alm1993upper} that, by discarding \emph{dead-end} $m$-step SAWs and assuming no parity restriction on the lattice, his matrix $\bG(m,n)$ and its dimensionally reduced version $\tilde{\bG}(m,n)$ (by construction) will both be positive for all sufficiently large $n$.
    However, it does not seem easy to quantify how large $n$ has to be relative to $m$ for this conclusion to hold.
    Hence, we have decided to make the primitivity of $\bG^\cP(m,n)$ an explicit assumption in Theorem \ref{thm:upper-bound}, which should be checked in applications.
    In practice, this may be accomplished, for example, by checking for positivity in small powers ($1,2,\dots$) of $\bG^\cP(m,n)$ when convenient, or by using the following equivalent characterization of primitivity \cite[Corollary 8.5.8]{horn2012matrix},
    \begin{equation}
    \label{eqn:primitivity-iff}
        \text{an }n\times n\text{ nonnegative matrix }A\text{ is primitive if and only if }A^{n^2-2n+2}\text{ is positive},
    \end{equation}
    for larger $\bG^\cP(m,n)$.
    We note that the primitivity assumption can only be violated due to the topology of the SAWs or the lattice, not the weighting scheme, since all the components of $\bz$ are strictly positive by assumption.
\end{remark}

\section{Results}
\label{sec:data}

We implemented the method described in Section \ref{sec:method} as a program in Python 3.
We designed the program to be modular and extensible, capable of analyzing SAWs and SATs on arbitrary lattices.
Specific lattices are entered manually as separate modules, including details such as the lattice structure, the weighting scheme $\bz$, and the associated weight-preserving symmetries, and discovered dynamically at runtime.
In particular, even though the full group of weight-preserving symmetries is infinite, as it includes all the translations, the program only requires the coset representatives modulo the translations.
The structure of the main program follows the presentation in Section \ref{sec:algorithm} and comprises three main steps: symbolic matrix construction, numerical evaluation and visualization, and validation.

\subsection{Symbolic matrix construction}

By inputting $\bz$ as a vector of formal variables, we constructed the matrix $\bG^\cP(m,n)$ from \eqref{eqn:G-matrix} symbolically using the \pythoninline{SymPy} library as follows.

First, for a given $m$, we generated a complete set of $m$-step SAWs (or SATs) starting from each of the $K$ vertex class representatives using an iterative, breadth-first algorithm.
For each vertex class representative $v_k$ and $0\le \ell\le m-1$, the algorithm attempts to extend each already-found $\ell$-step SAW starting from $v_k$ to a $(\ell+1)$-step SAW with the same starting point by appending a valid edge. 
Specifically, at each step, the algorithm queries a lattice-specific method \pythoninline{get_neighbor_vectors()} to determine the set of available outgoing edges from the walk's current endpoint, thus accounting for the local topology on lattices with multiple vertex classes. 
The $m$-step SAWs generated for each of the $K$ representatives are then pooled into a single, combined set.

Second, we partitioned this combined set of $m$-step SAWs into equivalence classes under the pre-specified weight-preserving symmetries.
In doing so, we also obtained a canonical representation for each generated $m$-step SAW by applying all weight-preserving symmetries and choosing the lexicographically smallest one as its canonical form.
The number of equivalence classes defines the dimension of the to-be-constructed matrix $\bG^\cP(m,n)$.

Third, we populated the entries of $\bG^\cP(m,n)$ via a recursive, depth-first search.
Starting with the canonical representation of an equivalence class $\cP_r(m)$ of $m$-step SAWs, the algorithm explores all valid extensions of it to an $n$-step SAW via a recursive, depth-first search.
At each step, the algorithm again calls upon the \pythoninline{get_neighbor_vectors()} method for the current endpoint to account for the local topology of the lattice.
For each extension it finds, the algorithm identifies its last $m$ steps as its \emph{tail} and computes the canonical form of the tail to determine the index of the equivalence class $\cP_s(m)$ to which it belongs.
The symbolic weight of the full $n$-step SAW is then added to the $(r,s)$-entry of the matrix.

Finally, in accordance with \eqref{eqn:G-matrix-element}, we divided each entry of the symbolic matrix by the weight of its corresponding $m$-step SAW, completing the construction of $\bG^\cP(m,n)$.
The final matrix is saved using \pythoninline{pickle} to facilitate further analysis.

\subsection{Numerical evaluation and visualization}
\label{sec:evaluation}

To derive upper bounds on the weighted connective constant $\bmu$, we numerically evaluated the dominant eigenvalue of $\bG^\cP(m,n)$.
To facilitate the evaluation, we converted the symbolic expression of the matrix into a numerical function using \pythoninline{sympy.lambdify()}.
In all the computational cases considered below, we verified the primitivity of $\bG^\cP(m,n)$ using \eqref{eqn:primitivity-iff}.

\subsubsection{Square lattice}

We studied both weighted SAWs and SATs on the square lattice, using (the coset representatives of) the weight-preserving symmetries $(n_1,n_2)\mapsto(\pm n_1,\pm n_2)$, where the signs are chosen independently.
We adopted the convention that $x$ and $y$ are the weights assigned to each horizontal and vertical edge.

For small values of $m$, the symbolic matrix $\bG^\cP(m,n)$ is sufficiently small to be diagonalized exactly, which yields closed-form expressions of upper bounds on $\bmu$; see Table \ref{tab:closed-form_square}.
The expressions grow quickly in complexity with increasing values of $n$.

\begin{table}
\centering
\begin{tabular}{ccc}
\toprule
Object & $(m,n)$ & Upper bound \\ 
\midrule
SAW,SAT & $(1,2)$ & $\frac{1}{2}\left(x+y+\sqrt{x^2+14xy+y^2}\right)$ \\
SAW,SAT & $(1,3)$ & $\frac{1}{2^{1/2}}\left[x^2+8xy+y^2+(x+y)\sqrt{x^2+14xy+y^2}\right]^{1/2}$ \\
SAW & $(1,4)$ & $\begin{multlined}
    \tfrac{1}{2^{1/3}}\left[x^3+12xy(x+y)+y^3+\vphantom{\sqrt{x^6}}\right. \\ 
    \left.\sqrt{x^6+24x^5y+136x^4y^2+254x^3y^3+136x^2y^4+24xy^5+y^6}\right]^{1/3}
    \end{multlined}$ \\
\addlinespace 
SAT & $(1,4)$ & $\frac{1}{2^{1/3}}\left[x^3+12xy(x+y)+y^3+(x^2+5xy+y^2)\sqrt{x^2+14xy+y^2}\right]^{1/3}$ \\
\bottomrule
\end{tabular}
\caption{Upper bounds on the weighted connective constant $\bmu$ of self-avoiding walks (SAWs) and self-avoiding trails (SATs) on the square lattice}
\label{tab:closed-form_square}
\end{table}

We resorted to numerical evaluation for more general pairs $(m,n)$.
In Figure \ref{fig:square}, we computed the dominant eigenvalue of the numerical matrix on a $100\times 100$ grid in the parameter regime $0\le x,y\le 1$, and plotted the contour where $\lambda_1(\bG^\cP(m,n))=1$.
The relative positions of the contours indicate that our bounds become progressively tighter as $m$ and $n$ increase, agreeing with Alm's observation in the unweighted case \cite[Conjecture 1]{alm1993upper}.
This trend is most visible in the inset plots, where we magnified the contours around the line $y=x$, which corresponds to having isotropic weights.

\begin{figure}
    \centering
    \begin{subfigure}[t]{0.49\textwidth}
        \includegraphics[page=1,width=\linewidth]{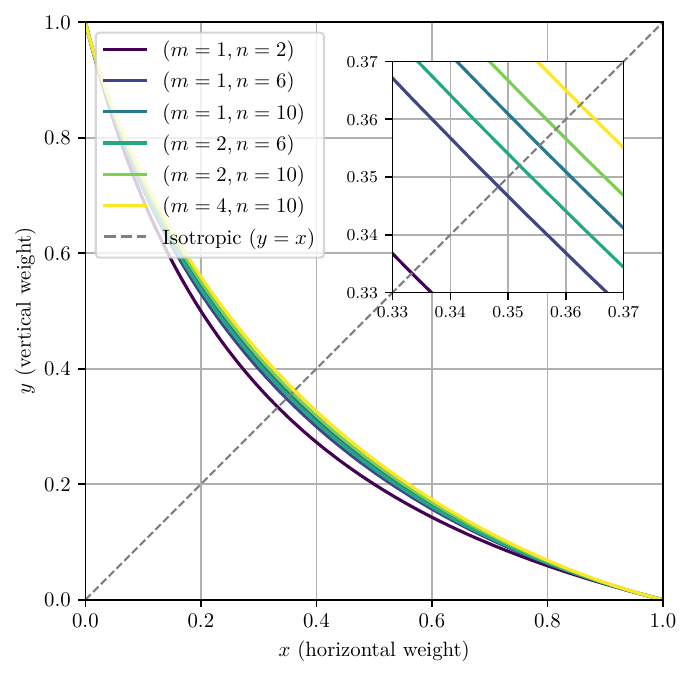}
        \caption{Self-avoiding walks}
        \label{fig:square_walk}
    \end{subfigure}
    \hfill
    \begin{subfigure}[t]{0.49\textwidth}
        \includegraphics[page=2,width=\linewidth]{plots.pdf}
        \caption{Self-avoiding trails}
        \label{fig:square_trail}
    \end{subfigure}
    \caption{Plots of the curves $\lambda_1(\bG^\cP(m,n))=1$ for the square lattice}
    \label{fig:square}
\end{figure}

\subsubsection{Simple cubic lattice}
\label{sec:simple-cubic}

To study weighted SAWs and SATs on the simple cubic lattice, we adopted the convention that steps in the $\pm\vb{e}_1,\pm\vb{e}_2,\pm\vb{e}_3$ directions are weighted by $x,y,z$, respectively.

Under the most general weighting scheme, where no relation between $x,y,z$ is imposed, we used the weight-preserving symmetries $(n_1,n_2,n_3)\mapsto(\pm n_1,\pm n_2,\pm n_3)$, where the signs are independently chosen.
Here, even for $m=1$, exact diagonalization of the symbolic matrix is already impractical, albeit technically possible using Cardano's formula, so we resorted again to numerical evaluation.
In Figure \ref{fig:simple-cubic}, we plot the isosurface $\lambda_1(\bG^\cP(m,n))=1$, generated by computing the dominant eigenvalue on a three-dimensional grid of weights and applying the marching cubes algorithm implemented in the \pythoninline{scikit-image} library.
Figures \ref{fig:simple-cubic_walk_full} and \ref{fig:simple-cubic_trail_full} show a full view of the isosurfaces.
The visibility of exactly one isosurface indicates that its corresponding upper bound is, for all $x,y,z>0$, the best among all that appear in the plot.
To overcome the occlusion, Figures \ref{fig:simple-cubic_walk_local} and \ref{fig:simple-cubic_trail_local} show a local view of the isosurfaces near the isotropic line $x=y=z$, where the relative positions of the isosurfaces are clearly visible and again suggest that the bounds improve with larger values of $m$ and $n$.

\begin{figure}
    \centering
    \begin{subfigure}[t]{0.5\textwidth}
        \includegraphics[page=3,width=\linewidth]{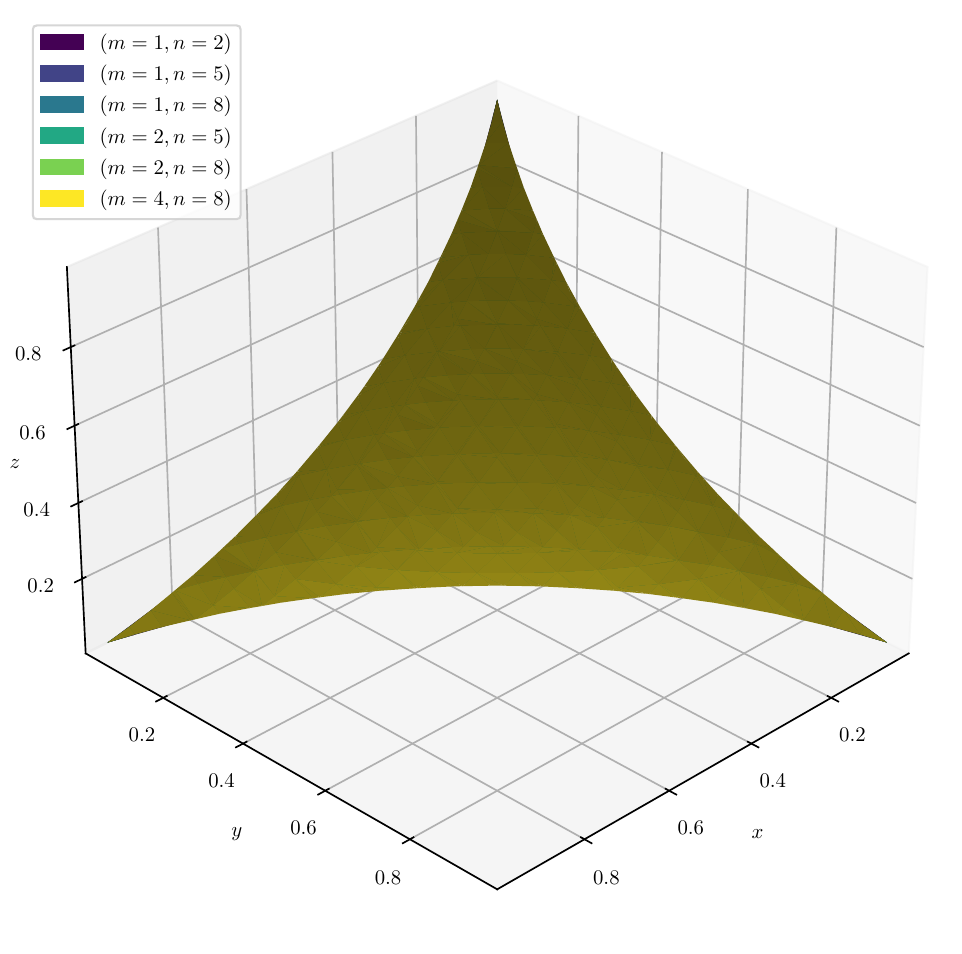}
        \caption{Self-avoiding walks: full view}
        \label{fig:simple-cubic_walk_full}
    \end{subfigure}%
    \hfill
    \begin{subfigure}[t]{0.5\textwidth}
        \includegraphics[page=4,width=\linewidth]{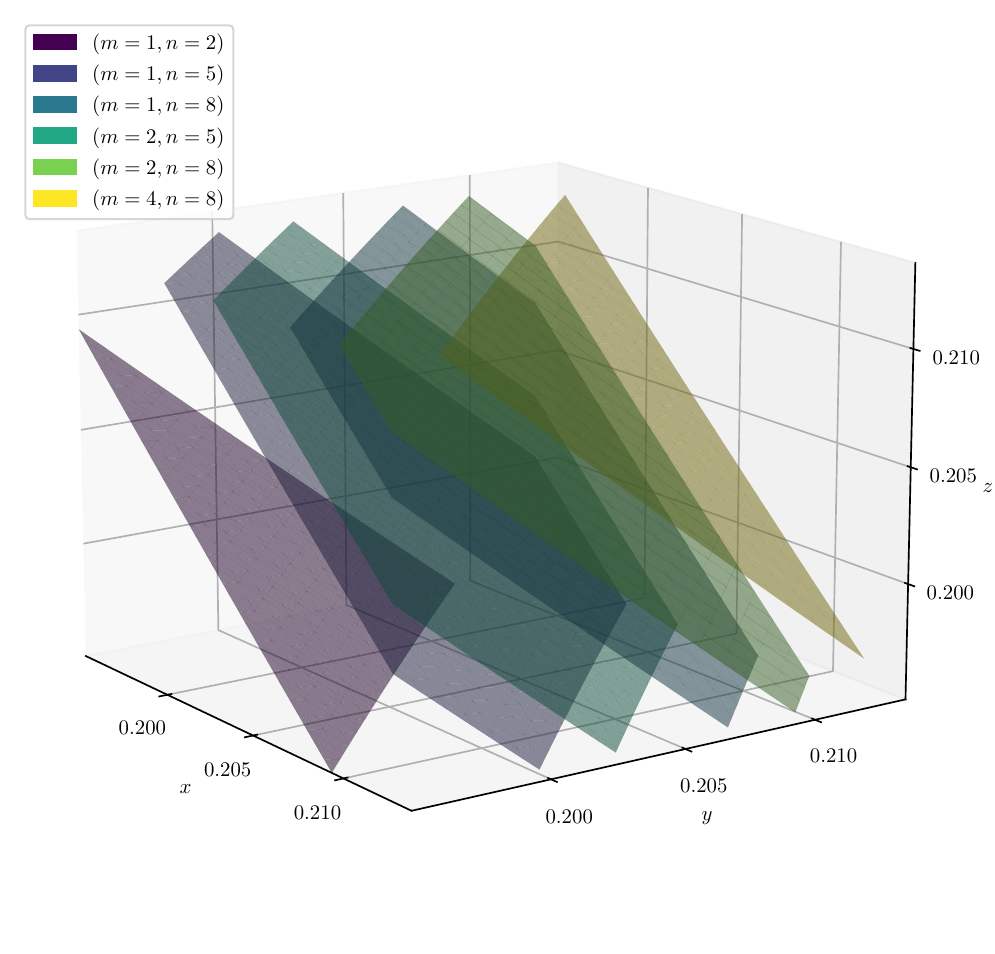}
        \caption{Self-avoiding walks: local view}
        \label{fig:simple-cubic_walk_local}
    \end{subfigure}
    \bigskip
    \begin{subfigure}[t]{0.5\textwidth}
        \includegraphics[page=5,width=\linewidth]{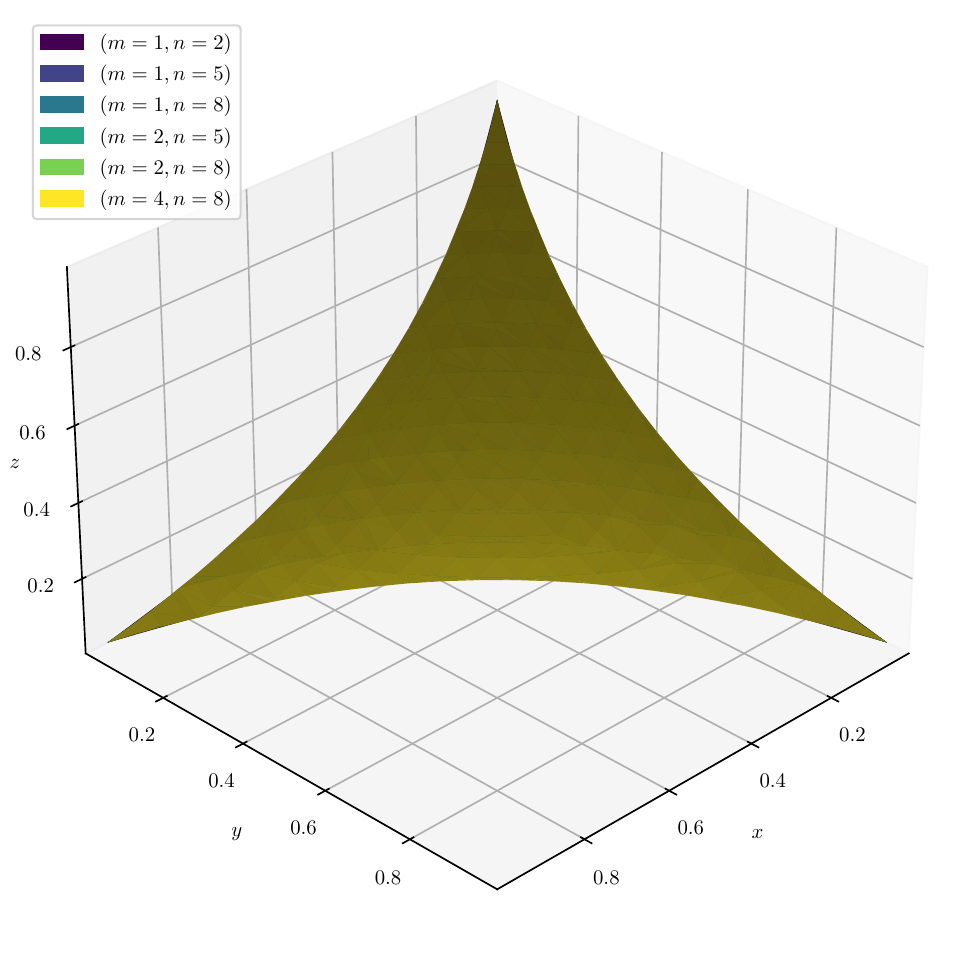}
        \caption{Self-avoiding trails: full view}
        \label{fig:simple-cubic_trail_full}
    \end{subfigure}%
    \hfill
    \begin{subfigure}[t]{0.5\textwidth}
        \includegraphics[page=6,width=\linewidth]{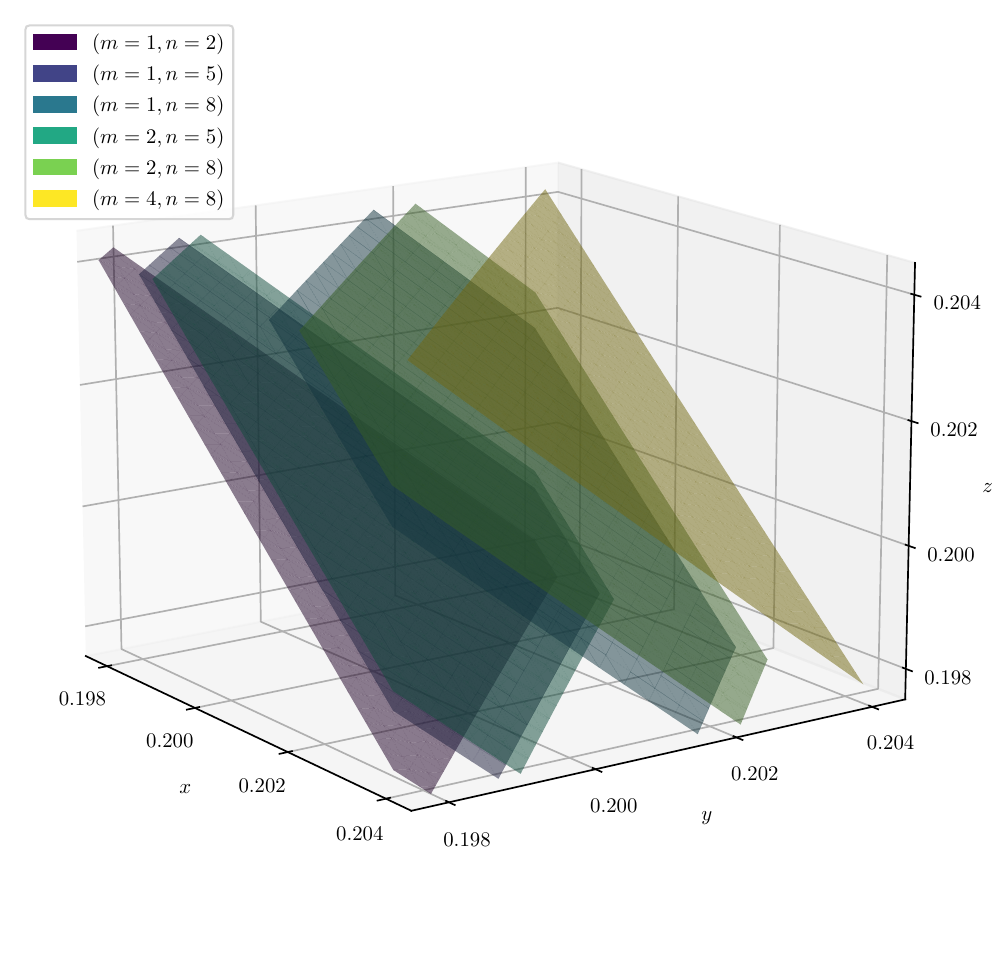}
        \caption{Self-avoiding trails: local view}
        \label{fig:simple-cubic_trail_local}
    \end{subfigure}
    \caption{Plots of the isosurfaces $\lambda_1(\bG^\cP(m,n))=1$ for the simple cubic lattice}
    \label{fig:simple-cubic}
\end{figure}

We also considered the special case where $x=y$ and $z$ remains independent.
The weight-preserving symmetries are given by $(n_1,n_2,n_3)\mapsto(\pm n_1,\pm n_2,\pm n_3)$ and $(n_1,n_2,n_3)\mapsto(\pm n_2,\pm n_1,\pm n_3)$.
With this reduction, closed-form expressions of upper bounds on $\bmu$ are readily accessible for both SAWs and SATs for small values of $m$; see Table \ref{tab:closed-form_simple-cubic-reduced}.

\begin{table}
\centering
\begin{tabular}{ccc}
\toprule
Object & $(m,n)$ & Upper bound \\ 
\midrule
SAW,SAT & $(1,2),(2,3)$ & $\tfrac{1}{2}\left(3x+z+\sqrt{9x^2 + 26 x z + z^2}\right)$ \\
SAW,SAT & $(1,3)$ & $\tfrac{1}{2^{1/2}}\left[9 x^2 + 16 x z + z^2 + (3x+z)\sqrt{9x^2 + 26 x z + z^2}\right]^{1/2}$ \\
\bottomrule
\end{tabular}
\caption{Upper bounds on the weighted connective constant $\bmu$ of self-avoiding walks (SAWs) and self-avoiding trails (SATs) on the simple cubic lattice, under the constraint that steps in the $\pm\vb{e}_1$ and $\pm\vb{e}_2$ directions are weighted equally}
\label{tab:closed-form_simple-cubic-reduced}
\end{table}

\subsubsection{Triangular lattice}

To study weighted SAWs and SATs on the triangular lattice, we adopted Alm's convention \cite[Section 5]{alm1993upper} to position the vertices of the triangular lattice on $\Z^2$, whereby the origin is connected to the six points $(0, 1)$, $(0, -1)$, $(1, 0)$, $(-1, 0)$, $(1, 1)$, and $(-1, -1)$.
In the special case where steps in the $\pm\vb{e}_1$ and $\pm\vb{e}_2$ directions are weighted equally by $x$ and those in the $\pm(1,1)$ directions by $z$, using the weight-preserving symmetries $(n_1,n_2)\mapsto(n_1,n_2),(-n_1,-n_2),(n_2,n_1),(-n_2,-n_1)$, we again obtained closed-form expressions of upper bounds on $\bmu$; see Table \ref{tab:closed-form_triangular-reduced}.

\begin{table}
\centering
\begin{tabular}{ccc}
\toprule
Object & $(m,n)$ & Upper bound \\ 
\midrule
SAW,SAT & $(1,2)$ & $\tfrac{1}{2}\left(3x+z+\sqrt{9x^2 + 26 x z + z^2}\right)$ \\
SAW & $(1,3)$ & $\tfrac{1}{2^{1/2}}\left(9 x^2 + 15 x z + z^2 + \sqrt{
 81 x^4 + 182 x^3 z + 143 x^2 z^2 + 34 x z^3 + z^4}\right)^{1/2}$ \\
SAT & $(1,3)$ & $\tfrac{1}{2^{1/2}}\left[9 x^2 + 16 x z + z^2 + (3x+z)\sqrt{9 x^2 + 26 x z + z^2}\right]^{1/2}$ \\
\bottomrule
\end{tabular}
\caption{Upper bounds on the weighted connective constant $\bmu$ of self-avoiding walks (SAWs) and self-avoiding trails (SATs) on the (tilted) triangular lattice, under the constraint that steps in the $\pm\vb{e}_1$ and $\pm\vb{e}_2$ directions are weighted equally}
\label{tab:closed-form_triangular-reduced}
\end{table}

\subsubsection{Hexagonal lattice}

Lastly, to give an example of a lattice with more than one vertex class, we considered weighted SAWs and SATs on the hexagonal lattice, which has $K=2$.
Like for the triangular lattice, it is convenient to position the vertices of the hexagonal lattice on $\Z^2$.
Specifically, $(0,0)$ represents one vertex class and is connected to the three vertices $(1,0)$, $(0,1)$, and $(-1,-1)$; $(1,0)$ represents the other vertex class and is connected to $(0,0)$, $(1,-1)$, and $(2,1)$; and the remainder of the lattice is generated by the translations by $(2,1)$ and $(1,-1)$.

Under the most general weighting scheme, where edges in the direction of $(1,0)$, $(0,1)$, and $(1,1)$ are weighted respectively by $x$, $y$, and $z$, the lattice admits the weight-preserving symmetries $(x,y)\mapsto(x,y),(1-x,-y)$.
In the special case where $x=y$, using the same weight-preserving symmetries, we obtained closed-form expressions of upper bounds on $\bmu$; see Table \ref{tab:closed-form_hexagonal-reduced}.

\begin{table}
\centering
\begin{tabular}{ccc}
\toprule
Object & $(m,n)$ & Upper bound \\ 
\midrule
SAW,SAT & $(1,2)$ & $\tfrac{1}{2}\left[x+\sqrt{x(x+8z)}\right]$ \\
SAW,SAT & $(1,3)$ & $\tfrac{1}{2^{1/2}}\left[x^2+4xz+x\sqrt{x(x+8z)}\right]^{1/2}$ \\
SAW,SAT & $(1,4)$ & $\tfrac{1}{2^{1/3}}\left[x^3+6x^2z+x(x+2z)\sqrt{x(x+8z)}\right]^{1/3}$ \\
\bottomrule
\end{tabular}
\caption{Upper bounds on the weighted connective constant $\bmu$ of self-avoiding walks (SAWs) and self-avoiding trails (SATs) on the (tilted) hexagonal lattice, under the constraint that steps in the $\pm\vb{e}_1$ and $\pm\vb{e}_2$ directions are weighted equally}
\label{tab:closed-form_hexagonal-reduced}
\end{table}

\subsection{Validation}

We validated our results against Alm's upper bounds on the connective constant of unweighted SAWs using the following observation: for all $c>0$, 
\begin{equation}
\label{eqn:scaling}
    \lambda_1(\bG^\cP_{c\bz}(m,n))=c^{n-m}\lambda_1(\bG^\cP_{\bz}(m,n)),
\end{equation} 
where the subscript on $\bG$ indicates the weight with respect to which the matrix is defined, and which is an immediate consequence of \eqref{eqn:G-matrix-element}.
Writing $\vb{1}:=(1,1,\dots,1)\in\R^d$, where $d$ is as defined in Section \ref{sec:notation}, it follows from \eqref{eqn:scaling} that if $\lambda_1(\bG^\cP_{c\vb{1}}(m,n))^{1/(n-m)}=1$, then 
\begin{equation}
\label{eqn:reciprocal}
    \lambda_1(\bG^\cP_{\vb{1}}(m,n))^{1/(n-m)}=c^{-1}.
\end{equation}
Now, by a straightforward adaptation of \cite[Theorem 2]{alm1993upper}, we know that our matrix $\bG^\cP_{\vb{1}}(m,n)$ has the same dominant eigenvalue as Alm's $\bG(m,n)$, provided that both matrices are primitive.
As we have verified the primitivity of $\bG^\cP_{\vb{1}}(m,n)$ in all our cases, and Alm has noted that so is the case for his $\bG(m,n)$ \cite[Remark 5]{alm1993upper}, the identity \eqref{eqn:reciprocal} implies that the $x$-coordinate of the intersection of the contour or isosurface $\lambda_1(\bG^\cP(m,n))=1$ with the isotropic line $x=y$ (for the square lattice) or $x=y=z$ (for the simple cubic lattice and the triangular lattice) is equal to the reciprocal of Alm's upper bound on the unweighted connective constant for the same $(m,n)$.
The closed-form expressions in Tables \ref{tab:closed-form_square}--\ref{tab:closed-form_hexagonal-reduced} indeed agree in this sense with the bounds given in \cite[Tables 3, 4, 5, and 9]{alm1993upper}.

\section{Discussion}

\subsection{Convergence of the anisotropic generating function}
\label{sec:generating-function}

As observed in \cite[Proposition 2.2(iii)]{borgs2000anisotropic}, upper bounds on the weighted connective constant $\bmu$ have immediate implications for the domain of convergence of the \emph{anisotropic generating function} of SAWs and SATs.
We formulate this formally as follows.

Let $\bx:=(\bx_1,\dots,\bx_d)$ be a $d$-tuple of formal variables.
For $\bn\in\Z_{\ge 0}^d$ and $1\le k\le K$, we introduce the shorthand $\bx^\bn:=\prod_{i=1}^d \bx_i^{\bn_i}$, and recall from Section \ref{sec:notation} the definition of $c_{\bn,k}$.
The anisotropic generating function of SAWs (or SATs) on $\cL$ with starting point in $\cV_k$ is defined as the formal power series
\begin{equation}
\label{eqn:generating-function}
    g_k(\bx):=\sum_{\bn\in\Z_{\ge 0}^d} c_{\bn,k} \bx^\bn;
\end{equation}
see \cite[(1.4)]{borgs2000anisotropic} or \cite{rechnitzer2006haruspicy}.

\begin{corollary}
\label{cor:generating-function}
    For all $1\le k\le K$ and $0\le m<n$, the anisotropic generating function $g_k(\bx)$ converges absolutely on the open region
    \begin{equation}
        \cD^\cP(m,n):=\bigcup_{\substack{\bz\in\R_{>0}^d\\\lambda_1(\bG^\cP(m,n))<1}}[-\bz,\bz],
    \end{equation}
    where we abbreviated $[-\bz,\bz]:=\prod_{i=1}^d[-\bz_i,\bz_i]$.
    In particular, $g_k(\bx)$ is analytic on $\cD^\cP(m,n)$.
\end{corollary}

\begin{proof}
    The openness of $\cD^\cP(m,n)$ follows straightforwardly from \eqref{eqn:scaling}.
    To prove convergence, let $\bz\in\R_{>0}^{d}$ be such that $\lambda_1(\bG^\cP(m,n))<1$.
    Since $g_k(\bz)$ contains only nonnegative terms, we can rearrange
    \begin{equation}
        g_k(\bz)=\sum_{\ell=0}^\infty\sum_{\sum_{i=1}^d \bn_i=\ell}c_{\bn,k}\bz^\bn
        =\sum_{\ell=0}^\infty\bc_{\ell,k}.
    \end{equation}
    By Theorem \ref{thm:upper-bound}, 
    \begin{equation}
        \lim_{\ell\to\infty}\bc_{\ell,k}^{1/\ell}=\bmu\le\lambda_1(\bG^\cP(m,n))<1,
    \end{equation}
    which implies that $g_k(\bz)<\infty$.
    By the comparison test, $g_k(\bx)$ converges absolutely on $[-\bz,\bz]$, which proves the first claim.
    The second claim is a standard fact from the theory of several complex variables; see \cite[Theorem 2.4.2]{hormander1990introduction}.
\end{proof}

\subsection{Contour models with anisotropic weights}
\label{sec:contour-model}

As an application of Corollary \ref{cor:generating-function}, we introduce and study now a \emph{toy} contour model on $\Z^2$, where each contour is assigned a weight via a (strongly) \emph{anisotropic} rule.

Before discussing motivations, let us first introduce the model.
Recall from Section \ref{sec:notation} the definition of a self-avoiding trail (SAT) $(v_0,e_1,v_1,\dots,e_n,v_n)$.
A \emph{circuit} is an SAT whose endpoints coincide: $v_0=v_n$.
To avoid trivial cases, we require that $n\ge 1$.
We consider an ensemble of circuits on $\Z^2$, as follows.
Given $\alpha,\epsilon\in\R$, the weight of a circuit $\gamma$ is defined as
\begin{equation}
\label{eqn:contour-model_weight}
    \bw(\gamma):=\epsilon^{N_h(\gamma)}\alpha^{N_v(\gamma)},
\end{equation}
where $N_h(\gamma)$ and $N_v(\gamma)$ denote the number of horizontal and vertical edges in $\gamma$.
The circuits interact via a hard-core pair interaction $\delta(\gamma,\gamma')$, which forbids $\gamma$ and $\gamma'$ from sharing an edge.
The formal partition function of the model is
\begin{equation}
\label{eqn:contour-model_partition-function}
    \Xi_{t,\epsilon}:=\sum_{\Gamma'\subseteq\Gamma}\prod_{\gamma\in\Gamma'}\bw(\gamma)\prod_{\set{\gamma,\gamma'}\subseteq\Gamma'}\delta(\gamma,\gamma'),
\end{equation}
where $\Gamma$ denotes the set of all circuits on $\Z^2$.
The definition \eqref{eqn:contour-model_partition-function} is adapted to finite volumes in the standard way.

To motivate this model, it is helpful to recall the contour representation of the Ising model on $\Z^2$ at low temperature and zero magnetic field \cite[Section 3.7.2]{friedli2018statistical}.
There, expecting neighboring pairs of opposite spins to be energetically unfavorable, one constructs \emph{contours} on the dual graph of $\Z^2$ which delineate interfaces between contiguous regions of $+$ and $-$ spins.
Each contour $\gamma$ carries a penalty factor proportional to $e^{-2\beta\abs{\gamma}}$, where $\beta>0$ is the inverse temperature and $\abs{\gamma}$ is the length of $\gamma$.
Then, a \emph{Peierls argument} proves the existence of two stable phases at large $\beta$, when the quantity $e^{-2\beta}$ overwhelms a finite combinatorial factor, often taken to be the growth rate of the number of \emph{non-backtracking walks} on the lattice \cite[(3.39)]{friedli2018statistical}, which bounds the entropy of contours.
The estimates in this argument are \emph{isotropic} in the sense that they do not explicitly distinguish between horizontal and vertical steps in a contour.
Then, it is natural to wonder what would happen if one were to differentiate between steps in different directions: \emph{would a Peierls argument still apply?}

This question is no mere mathematical curio.
In their recent study of Glauber dynamics for the six-vertex model \cite{fahrbach2019slow}, Fahrbach--Randall encountered \emph{fault lines} in the antiferroelectric phase of the model, whose weights are determined by an \emph{anisotropic} rule.
They provided an affirmative answer to the question in their context, by showing that a Peierls-type argument applies if one uses the generating function of \emph{weighted non-backtracking walks} to control the energy-entropy competition directly, as opposed to treating the energetic and entropic aspects separately in the case of the Ising model \cite[Section 4.2]{fahrbach2019slow}.
Incidentally, weighted non-backtracking walks are precisely what we have needed to count to evaluate $\lambda_1(\bG^\cP(1,2))$ for SAWs on the square lattice, which explains the appearance of the expression in Table \ref{tab:closed-form_square} in their calculations \cite[Lemma 4.6]{fahrbach2019slow}.
From this point of view, the method we have developed, in the context of enabling the development of Peierls-type arguments for general contour models, generalizes the technique of Fahrbach--Randall by 
\begin{itemize}
    \item applying to anisotropic contour models on lattices with any finite number $K$ of vertex classes: the square lattice on which the six-vertex model is defined has $K=1$;
    \item allowing \emph{tighter} control over the energy-entropy competition using better bounds on $\bmu$, computed using larger values of $m$ and $n$: the argument by Fahrbach--Randall corresponds to taking $(m,n)=(1,2)$. 
\end{itemize}
For instance, by adapting the arguments in \cite[Section 4.2]{fahrbach2019slow} and using the other expressions for SAWs in Table \ref{tab:closed-form_square}, one should be able to deduce slow mixing of the Glauber dynamics for a \emph{larger} set of parameters in the antiferroelectric regime of the six-vertex model.

Another context where anisotropic contours arise is the high-activity columnar phase in the hard-square model on $\Z^2$.
Due to the phenomenon of \emph{deconfinement of half vacancies} \cite{ramola2012high}, also known as \emph{sliding} \cite{jauslin2017crystalline,mazel2019high,mazel2025high,hadas2025columnar}, \emph{defects} in the model exhibit strong anisotropy at high activities; see the example in Figure \ref{fig:rod}, constructed with $2\times 2$ squares. 
There, the four red squares at the center make up a single defect, called a \emph{rod} of length $4$ in the language of Ramola--Dhar \cite{ramola2012high}, over a background of row-ordered squares.
Keeping in mind that large activities only directly penalize the presence of vacancies, it is clear that the penalty does not equally affect the horizontal and vertical boundaries of the rod.
Indeed, the vacancies a rod opens up are restricted to its ends, while its (vertical) length is modulated by a much weaker entropic penalty due to its interruption of otherwise ordered rows \cite[(13)]{nath2016stability}.
In part due to this anisotropy, proving mathematically the existence of the columnar phase remains an open problem, except in the case of $2\times2$ squares, for which Hadas--Peled have recently given a proof based on reflection positivity \cite{hadas2025columnar}.
Notably, they showed that the \emph{usual} Peierls argument still applies to contours defined on a \emph{mesoscopic} scale \cite[Theorem 5.2]{hadas2025columnar}.

\begin{figure}
    \centering
    \includegraphics[width=0.5\linewidth]{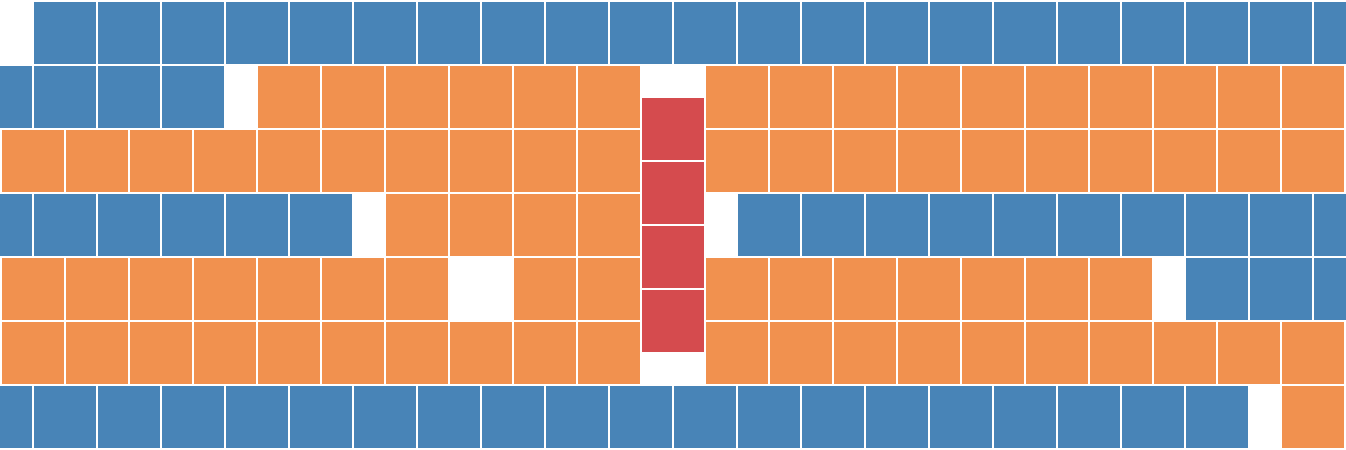}
    \caption{The red squares at the center make up a single defect over a background of row-ordered squares.
    The heavy penalty for vacancies in the high-activity regime is only directly associated with the rod's vertical ends, not its (vertical) length.
    This mechanism results in the formation of anisotropic defects at high activities
    }
    \label{fig:rod}
\end{figure}

Such systems directly inspire the structure of our toy model. 
We aim to capture the essential features of the anisotropy---a high cost for certain step directions and a low cost for others---in a simple mathematical form, which has led us to the weight function \eqref{eqn:contour-model_weight}.
To keep things simple yet interesting, we will treat $\epsilon$ as a perturbation variable and allow $\alpha$ to depend on $\epsilon$.
We will use Table \ref{tab:closed-form_square} and Corollary \ref{cor:generating-function} to prove that if the vertical component of the circuit is \emph{subcritical} in the $\epsilon\to0$ limit, then, for all small enough $\abs{\epsilon}$, our toy model admits a convergent \emph{cluster expansion} \cite{mayer1937statistical}.
It is well-known that the cluster expansion delivers robust control over the correlation functions of the model \cite[Chapter 5]{friedli2018statistical}, but we will content ourselves with proving its convergence.
We leave potential applications of our method to the hard-square and related models for future work.

\begin{proposition}
\label{prop:cluster-expansion}
    Let $f(\epsilon)$ be defined in a neighborhood of $0$.
    Suppose that $f$ is continuous at $0$ with $\abs{f(0)}<1$.
    Then, there exists $\epsilon_0>0$ such that, for all $\epsilon\in[-\epsilon_0,\epsilon_0]$, the model \eqref{eqn:contour-model_partition-function} with $\alpha:=f(\epsilon)$ satisfies the Koteck\'y--Preiss criterion \cite{kotecky1986cluster} for the convergence of the cluster expansion: defining $a:\Gamma\rightarrow\R_{>0}$ by
    \begin{equation}
    \label{eqn:cluster-expansion_a}
        a(\gamma):=t\left[N_h(\gamma)+N_v(\gamma)\right],
    \end{equation}
    where $t>0$ is any fixed constant such that $\abs{f(0)}e^t<1$, we have that, for all $\gamma_\ast\in\Gamma$,
    \begin{equation}
    \label{eqn:cluster-expansion_criterion}
        \sum_{\gamma\in\Gamma}\abs{\bw(\gamma)}e^{a(\gamma)}\abs{\delta(\gamma,\gamma_\ast)-1}\le a(\gamma_\ast).
    \end{equation}
\end{proposition}

\begin{proof}
    Let $\epsilon$ be in the domain of definition of $f$.
    Let $a(\gamma)$ be as in \eqref{eqn:cluster-expansion_a}.
    Notice that
    \begin{equation}
    \label{eqn:cluster-expansion_anchor}
        \sum_{\gamma\in\Gamma}\abs{\bw(\gamma)}e^{a(\gamma)}
        \abs{\delta(\gamma,\gamma_\ast)-1}
        \le [N_h(\gamma)+N_v(\gamma)]
        \max_{e_\ast\in\gamma_\ast}\sum_{\substack{\gamma\in\Gamma\\\gamma\ni e_\ast}}\abs{\bw(\gamma)}e^{a(\gamma)}.
    \end{equation}
    Let $e_\ast\in\gamma_\ast$.
    Since every circuit $\gamma$ containing $e_\ast$ is an SAT starting from a (fixed) endpoint of $e_\ast$, we have that
    \begin{equation}
    \label{eqn:cluster-expansion_domination}
        \sum_{\substack{\gamma\in\Gamma\\\gamma\ni e_\ast}}\abs{\bw(\gamma)}e^{a(\gamma)}
        \le \max_{v\in\Z^2}\sum_{\substack{\gamma\mathrm{\ SAT}\\\mathrm{starting\ from\ }v}}\abs{\bw(\gamma)}e^{a(\gamma)}
        =g\left(\abs{\epsilon}e^t,\abs{f(\epsilon)} e^t\right).
    \end{equation}
    In the above equality, we have adopted the convention that the first and second arguments of $g(\cdot,\cdot)$ correspond respectively to the horizontal and vertical steps of a trail (cf. \eqref{eqn:generating-function}), and suppressed the subscript on $g$ since the square lattice has $K=1$.
    Using the closed-form expression of $\lambda_1(\bG^\cP(1,2))$ in Table \ref{tab:closed-form_square}, as well as the assumption that $f$ is continuous at $0$ with $\abs{f(0)}<1$, it is easy to see that there exists $\epsilon_1>0$ such that,
    \begin{equation}
        \left(\abs{\epsilon}e^{t},\abs{f(\epsilon)} e^t\right)\in\cD^\cP(1,2)\quad\text{ for all }\epsilon\in[-\epsilon_1,\epsilon_1].
    \end{equation}
    By \eqref{eqn:cluster-expansion_domination} and Corollary \ref{cor:generating-function}, the sum
    \begin{equation}
        \sum_{\substack{\gamma\in\Gamma\\\gamma\ni e_\ast}}\abs{\bw(\gamma)}e^{a(\gamma)}
        =\sum_{\substack{\gamma\in\Gamma\\\gamma\ni e_\ast}}\left(\abs{\epsilon} e^{t}\right)^{N_h(\gamma)}\left(\abs{f(\epsilon)}e^t\right)^{N_v(\gamma)}
    \end{equation}
    converges on $[-\epsilon_1,\epsilon_1]$, is continuous at $0$, and, because every circuit contains at least two horizontal edges, vanishes at $0$.
    By continuity, there exists $\epsilon_0\in(0,\epsilon_1]$ such that the sum is less than $t$ on $[-\epsilon_0,\epsilon_0]$.
    Recalling \eqref{eqn:cluster-expansion_a} and \eqref{eqn:cluster-expansion_anchor}, we conclude that \eqref{eqn:cluster-expansion_criterion} holds for all $\epsilon\in[-\epsilon_0,\epsilon_0]$.
\end{proof}

\section{Conclusion}

Building on an earlier work by Alm \cite{alm1993upper}, we have developed a systematic method for deriving rigorous upper bounds on the connective constant for weighted self-avoiding walks (SAWs) and self-avoiding trails (SATs) on a general class of lattices, filling a methodological gap in the existing literature.
We have implemented our method in a publicly available software package, and applied it to weighted SAWs and SATs on the square lattice, the triangular lattice, the hexagonal lattice, and the simple cubic lattice, obtaining upper bounds as functions of the weights assigned to various classes of edges.
We have shown that such bounds imply the convergence of the anisotropic generating function for SAWs and SATs on specific domains in the parameter space.
Generalizing a technique of Fahrbach--Randall \cite{fahrbach2019slow}, we have further demonstrated that the knowledge of convergence allows us to prove Peierls-type estimates for contour models in statistical physics that feature anisotropic weights.

The present work leaves open several directions for future work.
For instance, one may try to derive rigorous \emph{lower bounds} on the connective constant for weighted SAWs by generalizing existing methods in the literature for unweighted SAWs \cite{kesten1963number,beyer1972lower}.
The other direction that we mention, which we find particularly exciting, is to explore the possibility of using our method to design Peierls-type arguments for concrete models in statistical physics, e.g., the hard-square model on $\Z^2$ \cite{hadas2025columnar}.

\subsection*{Acknowledgments}

The author thanks Joel Lebowitz for helpful background discussions and Ian Jauslin and Ron Peled for many enjoyable conversations about the hard-square and related models.
This work is partially supported by a Research \& Conference Travel Award from the School of Graduate Studies at Rutgers University.

\subsection*{Data availability}

The data that support the findings of this study are openly available at the following URL: \url{https://github.com/qhe28/weighted_connective_constant}.

\bibliographystyle{plain}
\bibliography{bibliography}

\appendix

\section{Proof of Theorem \ref{thm:limit-definition}}
\label{appx:Hammersley}

Following Hammersley \cite{hammersley1957percolation}, we will begin by using a sub-multiplicativity argument to obtain an equivalent definition of $\bmu$.
Then, we will relate the sums $\bc$ of weights of SAWs starting from adjacent vertices using surgeries on individual SAWs.
Lastly, by induction, we will relate the sums of weights of SAWs starting from vertices in different equivalence classes $\cV$, from which the theorem follows.
Compared to \cite{hammersley1957percolation}, the main complications are due to our use of weighted edges, which requires keeping a careful track throughout the proof.

First, observe that, for all $n_1,n_2\ge 1$, every $(n_1+n_2)$-step SAW is a \emph{concatenation} of an $n_1$- and an $n_2$-step SAW.
Hence,
\begin{equation}
\label{eqn:existence_concatenation}
    \bc_{n_1+n_2,k}\le\bc_{n_1,k}\bc_{n_2}\le\bc_{n_1}\bc_{n_2}\quad\text{ for all }1\le k\le K.
\end{equation}
Maximizing over $k$, we deduce that 
\begin{equation}
    \bc_{n_1+n_2}\le\bc_{n_1}\bc_{n_2}.
\end{equation}
By Fekete's lemma \cite[Theorem B.5]{friedli2018statistical},
\begin{equation}
\label{eqn:existence_Fekete}
    \bmu=\inf_{n\ge 1}\bc_{n}^{1/n}=\lim_{n\to\infty}\bc_n^{1/n}.
\end{equation}
It follows that, for all $\blambda>\bmu$, there exists a constant $C=C(\blambda)\ge 1$ such that 
\begin{equation}
\label{eqn:existence_exponential-growth}
    \bc_n\le C\blambda^{n}\quad\text{ for all }n\ge 1.
\end{equation}

Next, suppose that $v,v'$ are neighboring vertices of $\cL$, connected by an edge $e$, with $v\in\cV_k$ and $v'\in\cV_{k'}$.
Let $n\ge 1$.
Let $\gamma$ be a generic $2n$-step SAW starting from $v'$.
Introducing the shorthands $\bz_{\min}:=\min_{1\le i\le d}\bz_i$ and $\bz_{\max}:=\max_{1\le i\le d}\bz_i$, we have three disjoint cases:
\begin{itemize}
    \item $\gamma$ does not contain $v$.
    Then, $(v,e,\gamma)$ is a $(2n+1)$-step SAW starting from $v$, with 
    \begin{equation}
        \bw((v,e,\gamma))=\bw(e)\bw(\gamma)\ge\bz_{\min}\bw(\gamma).
    \end{equation}
    \item $\gamma$ first visits $v$ after $p$ steps (i.e., edges), where $1\le p\le 2n-1$.
    Let $\gamma_1$ and $\gamma_2$ be the self-avoiding sub-walks of $\gamma$ consisting of its first $p-1$ and last $2n-p$ steps, and $e_p$ be the $p$th edge of $\gamma$.
    Then, $(v,e,\gamma_1)$ and $\gamma_2$ are $p$- and $(2n-p)$-step SAWs starting from $v$, with 
    \begin{equation}
        \bw((v,e,\gamma_1))\bw(\gamma_2)=\frac{\bw(e)}{\bw(e_p)}\bw(\gamma)\ge\frac{\bz_{\min}}{\bz_{\max}}\bw(\gamma).
    \end{equation}
    \item $\gamma$ visits $v$ only at the very end, that is, after $2n$ steps.
    Let $\gamma_1$ be the self-avoiding sub-walk of $\gamma$ consisting of its first $2n-1$ steps, and $e_{2n}$ be the last edge of $\gamma$.
    Then, $(v,e,\gamma_1)$ is a $2n$-step SAW starting from $v$, with 
    \begin{equation}
        \bw((v,e,\gamma_1))=\frac{\bw(e)}{\bw(e_{2n})}\bw(\gamma)\ge\frac{\bz_{\min}}{\bz_{\max}}\bw(\gamma).
    \end{equation}
\end{itemize}
As the resulting (pairs of) SAWs in each case are distinct, we deduce that
\begin{equation}
\label{eqn:existence_change-starting-point}
    \bc_{2n,k'}\le \frac{1}{\bz_{\min}}\bc_{2n+1,k}+\frac{\bz_{\max}}{\bz_{\min}}\sum_{p=1}^{2n-1}\bc_{p,k}\bc_{2n-p,k}+\frac{\bz_{\max}}{\bz_{\min}}\bc_{2n,k}.
\end{equation}
Applying \eqref{eqn:existence_concatenation} and \eqref{eqn:existence_exponential-growth} to the RHS of \eqref{eqn:existence_change-starting-point}, we get that
\begin{equation}
\begin{multlined}
    \bc_{2n,k'}
    \le \frac{1}{\bz_{\min}}C\blambda^{n+1}\bc_{n,k}
    +\frac{\bz_{\max}}{\bz_{\min}}\left(\sum_{p=1}^{n}C\blambda^p\cdot C\blambda^{n-p}\bc_{n,k}
    +\sum_{p=n+1}^{2n-1}C\blambda^{p-n}\bc_{n,k}\cdot C\blambda^{2n-p}\right)
    \\
    +\frac{\bz_{\max}}{\bz_{\min}}C\blambda^n \bc_{n,k}
    =\blambda^n\bc_{n,k}\left[\frac{C\blambda}{\bz_{\min}}+\frac{C^2\bz_{\max}}{\bz_{\min}}(2n-1)+\frac{C\bz_{\max}}{\bz_{\min}}\right].
\end{multlined}
\end{equation}

By induction, for all walks $(v_0,e_1,v_1,\dots,e_q,v_q)$ on $\cL$, with $v_j\in\cV_{k_j}$ for each $0\le j\le q$, we have that
\begin{equation}
\label{eqn:existence_inductive-bound}
    \bc_{2^q n,k_q} \le \blambda^{(2^q-1)n}\bc_{n,k_0}
    \prod_{j=0}^{q-1}\left[\frac{C\blambda}{\bz_{\min}}+\frac{C^2\bz_{\max}}{\bz_{\min}}\left(2^{j+1}n-1\right)+\frac{C\bz_{\max}}{\bz_{\min}}\right]\quad\text{ for all }n\ge 1.
\end{equation}
Let $1\le k\le K$.
Let $Q_k\ge 0$ be the smallest integer such that, for all $1\le k'\le K$, there exists a walk of length $q_{k,k'}\le Q_k$ connecting a vertex in $\cV_k$ to a vertex in $\cV_{k'}$ (cf. Condition \ref{cond:strongly-connected}).
Combining \eqref{eqn:existence_concatenation}, \eqref{eqn:existence_exponential-growth}, and \eqref{eqn:existence_inductive-bound}, we get that
\begin{equation}
\begin{multlined}
    \bc_{2^{Q_k}n,k'}
    \le C\blambda^{(2^{Q_k}-2^{q_{k,k'}})n} \bc_{2^{q_{k,k'}}n,k'}
    \le C\blambda^{(2^{Q_k}-1)n} \bc_{n,k}
    \\
    \cdot\prod_{j=0}^{q_{k,k'}-1}\left[\frac{C\blambda}{\bz_{\min}}+\frac{C^2\bz_{\max}}{\bz_{\min}}\left(2^{j+1}n-1\right)+\frac{C\bz_{\max}}{\bz_{\min}}\right]\quad \text{ for all }1\le k'\le K,\ n\ge 1.
\end{multlined}
\end{equation}
Maximizing over $k'$, we have that
\begin{equation}
\begin{multlined}
    \bc_{2^{Q_k}n}
    \le C\blambda^{(2^{Q_k}-1)n} \bc_{n,k}
    \\
    \cdot\max_{1\le k'\le K}\prod_{j=0}^{q_{k,k'}-1}\left[\frac{C\blambda}{\bz_{\min}}+\frac{C^2\bz_{\max}}{\bz_{\min}}\left(2^{j+1}n-1\right)+\frac{C\bz_{\max}}{\bz_{\min}}\right]\quad \text{ for all }n\ge 1.
\end{multlined}
\end{equation}
By \eqref{eqn:existence_Fekete}, it follows that
\begin{equation}
\begin{split}
    \bmu{}&=\lim_{n\to\infty}\bc_{2^{Q_k}n}^{1/(2^{Q_k}n)}
    \\
    {}&\le \liminf_{n\to\infty}\Bigg\{C\blambda^{(2^{Q_k}-1)n} \bc_{n,k}
    \max_{1\le k'\le K}\prod_{j=0}^{q_{k,k'}-1}\left[\frac{C\blambda}{\bz_{\min}}+\frac{C^2\bz_{\max}}{\bz_{\min}}\left(2^{j+1}n-1\right)+\frac{C\bz_{\max}}{\bz_{\min}}\right]\Bigg\}^{1/(2^{Q_k}n)}
    \\
    {}&\le \blambda^{(2^{Q_k}-1)/2^{Q_k}}\left(\liminf_{n\to\infty}\bc_{n,k}^{1/n}\right)^{1/2^{Q_k}}.
\end{split}
\end{equation}
Rearranging the terms, we find that
\begin{equation}
    \left(\frac{\bmu}{\blambda}\right)^{2^{Q_k}}\blambda
    \le \liminf_{n\to\infty}\bc_{n,k}^{1/n}
    \le \lim_{n\to\infty}\bc_n^{1/n}
    =\bmu.
\end{equation}
Taking $\blambda\downarrow\bmu$ completes the proof.

\end{document}